\documentclass[12pt,twoside]{amsart}

\usepackage[utf8]{inputenc}      
\usepackage[T1]{fontenc}
\usepackage{lmodern}              
\usepackage{amsmath,amsthm,amssymb}
\usepackage{mathrsfs}
\usepackage{bm}
\usepackage{graphicx}
\usepackage{enumitem}
\usepackage{verbatim}
\usepackage{color}

\usepackage[a4paper,margin=1.2in]{geometry}

\usepackage{cite}                  

\usepackage{ulem}                
\usepackage{footmisc}

\newtheorem{Theorem}{Theorem}[section]
\newtheorem{definition}[Theorem]{Definition}
\newtheorem{lemma}[Theorem]{Lemma}

\newtheorem{Proposition}[Theorem]{Proposition}

\usepackage[colorlinks=true,linkcolor=black,citecolor=black,urlcolor=black]{hyperref}
\hypersetup{pdfauthor={Your Name}, pdftitle={Title}}

\begin{document}
\thanks{}

\address [Yutao Liu] {School of Mathematics (Zhuhai), Sun Yat-Sen University, Zhuhai, Guangdong,  519082, China}
\email{liuyt88@mail2.sysu.edu.cn}

\address [Jujie Wu] {School of Mathematics (Zhuhai), Sun Yat-Sen University, Zhuhai, Guangdong,  519082, China}
\email{wujj86@mail.sysu.edu.cn}

\address [Yuanpu Xiong] {School of Mathematical Sciences, Tongji University, Shanghai, 200092, China}
\email{ypxiong@tongji.edu.cn}

\title{The Grothendieck Theorem in Bergman Spaces }
\author{Yutao Liu, Jujie Wu and Yuanpu Xiong}

\date{}
\begin{abstract}  
	In this paper, we prove that if $E$ is a closed subspace of the holomorphic $L^p$-integrable space and is also contained in the holomorphic $L^q$-integrable space, for any $p > 1$ and any $q > p$, then the dimension of $E$ must be finite.
	

	
	\smallskip
	
\end{abstract}

\thanks{The Second author is supported by the National Key R\&D Program of China,  No.  2024YFA1015200 and the Natural Science Foundation of Guangdong Province,  No.  2025A1515011428.}

\maketitle
\tableofcontents
\section{Introduction}
In 1954, Grothendieck \cite{Gro} proved the following theorem (see also \cite{Stein}, Chapter 4, Theorem 4.2 and \cite{rudin1991functional}, Theorem 5.2).
\begin{Theorem}\label{th:Gro}
Let $(X,\mu)$ be a finite measure space and $1\leq p<\infty$. If $E$ is a closed subspace of $L^p(X,\mu)$ and $E\subset{L^\infty}(X,\mu)$, then $\dim{E}<\infty$.
\end{Theorem}
It is notable that the space $L^\infty(X,\mu)$ in the theorem cannot be replaced by $L^q(X,\mu)$ with $q>p$ (see \cite{rudin1991functional}, Theorem 5.3).

In this paper, we restrict our attention to Bergman spaces
\[
A^p(\Omega):=\mathcal{O}(\Omega)\cap{L^p(\Omega)},\ \ \ p\geq1,
\]
where $\Omega$ is a bounded domain in $\mathbb{C}^n$ with the Lebesgue measure. It follows that $A^p(\Omega)$ is a closed subspace of $L^p(\Omega)$, and hence it is a Banach space (see \S\,\ref{sec:preliminary} for more details). In sharp contrast to Theorem \ref{th:Gro}, we have the following

\begin{Theorem}\label{th:Gro1}
For any bounded domain $\Omega\subset\mathbb{C}^n$, if $1\leq p<q\leq\infty$, $E$ is a closed subspace of $A^p(\Omega)$ and $E\subset{A^q(\Omega)}$, then $\dim{E}<\infty$.
\end{Theorem}

Since $\Omega$ is of finite measure, the case that $q=\infty$ is a straightforward consequence of Theorem \ref{th:Gro}. Indeed, if $E$ is a closed subspace of $A^p(\Omega)$, then it is also closed in $L^p(\Omega)$. As $E\subset A^\infty(\Omega)\subset L^\infty(\Omega)$, Theorem \ref{th:Gro} applies. In what follows, we shall only consider the case $q<\infty$.

Since $A^q(\Omega)\subset A^p(\Omega)$ in view of H\"{o}lder's inequality, we have a natural inclusion
\[
\tau:A^q(\Omega)\rightarrow A^p(\Omega),\qquad \tau(f)=f.
\]
Our proof of Theorem \ref{th:Gro1} is based on the following observation, which might be of independent interest.

\begin{Theorem}\label{th:compact}
For any $1\leq p<q<\infty$, $\tau$ is a compact operator.
\end{Theorem}

When $\Omega$ is the unit ball, Zhu \cite{zhu2022embedding} and  Bao-Ma-Yan-Zhu \cite{bmyz} established necessary and sufficient conditions for the compactness of the embedding between weighted Bergman spaces. To handle the unweighted case for a general bounded domain, our approach is different.

\section{Preliminary}\label{sec:preliminary}
Let $\Omega$ be a bounded domain in $\mathbb{C}^n$, and let $\mathcal{O}(\Omega)$ denote the space of holomorphic functions on $\Omega$. For $1 \leq p < \infty$, define
\[
A^p(\Omega) := \left\lbrace f\in\mathcal{O}(\Omega)\ \big|\ \|f\|_p:=\int_\Omega|f|^p<+\infty \right\rbrace  ,
\]
the space of $L^p$-integrable holomorphic functions on $\Omega$. These spaces generalize the classical Bergman space $A^2(\Omega)$, which forms a Hilbert space. However, for $p \neq 2$, the lack of a complete orthonormal basis implies that $A^p(\Omega)$ is no longer a Hilbert space. Nevertheless, as will be established in Proposition \ref{Banach}, $A^p(\Omega)$ is a Banach space.

\begin{Proposition}[Bergman inequality, \texorpdfstring{\cite[Proposition 2.1]{chen2022p}}{}]\label{Bergman inequality}
For any compact set $S\subset\Omega$ and for all $f\in A^p(\Omega)$, there exists a constant $C_{S,\Omega}>0$ such that
$$\sup_S|f|^p\leq C_{S,\Omega}\|f\|^p_p.$$
\end{Proposition}

\begin{Proposition}[\texorpdfstring{\cite[Proposition 2.2]{chen2022p}}{}]\label{Banach}
$A^p(\Omega)$ is a Banach space for $p\geq1$.
\end{Proposition}
\begin{proof}
It suffices to show that $A^p(\Omega)$ is a closed subspace of $L^p(\Omega)$.

Let $\{f_j\} \subset A^p(\Omega)$ be a sequence in $L^p(\Omega)$ with $f_j \to f_0 \quad \text{in } L^p(\Omega) \text{ as } j \to \infty$. From the result of Proposition \ref{Bergman inequality}, the sequence $\{f_j\}$ is uniformly bounded on compact subsets of $\Omega$, which implies that $\{f_j\}$ forms a normal family. It follows from Montel's theorem that there exists a subsequence $\{f_{j_k}\}$ converging locally uniformly to some $\hat{f}_0 \in \mathcal{O}(\Omega)$. Fatou's lemma yields
\[
\int_\Omega |\hat{f}_0|^p \, dV \le \limsup_{k \to \infty} \|f_{j_k}\|_p^p < \infty,
\]
so that $\hat{f}_0 \in A^p(\Omega)$, and
\begin{align*}
\|f_{j_k} - \hat{f}_0\|_p 
&= \left( \int_\Omega \liminf_{m \to \infty} |f_{j_k} - f_{j_m}|^p \, dV \right)^{\frac{1}{p}} \\
&\le \liminf_{m \to \infty} \left( \int_\Omega |f_{j_k} - f_{j_m}|^p \, dV \right)^{\frac{1}{p}} \\
&= \liminf_{m \to \infty} \|f_{j_k} - f_{j_m}\|_p \\
&\le \liminf_{m \to \infty} \left( \|f_{j_k} - f_0\|_p + \|f_{j_m} - f_0\|_p \right) \\
&= \|f_{j_k} - f_0\|_p.
\end{align*}
It follows that $f_{j_k} \to \hat{f}_0$ in $L^p(\Omega)$. Thus, $f_0 = \hat{f}_0$ a.e. on $\Omega$, i.e., $f_0\in A^p(\Omega)$.
\end{proof}

The following elementary result in functional analysis plays an important role in proving our main theorems.

\begin{Theorem}[\texorpdfstring{\cite[pp.104]{rudin1991functional}}{}]\label{functional analysis}
For any Banach space $B$, the identity operator $I: B \rightarrow B$ is compact if and only if $B$ is finite-dimensional.
\end{Theorem}

Let $E$ be a closed subspace of $A^p(\Omega)$, which is Banach space equipped with the $L^p$-norm. Suppose that $E\subset A^q(\Omega)$. To prove the finite-dimensionality of $E$, it suffices to show that the natural inclusion operator
\begin{equation}\label{eq:natural_inclusion}
j: (E, \|\cdot\|_p) \longrightarrow A^p(\Omega),\ \ \ f\mapsto f
\end{equation}
is compact, in view of Theorem \ref{functional analysis}. We shall make use the Vitali convergence theorem, which is related to the concept of uniform integrability.

\begin{definition}[Uniform Integrability]
Let $(X,\mu)$ be a measure space, $1\leq p<+\infty$. Let $\{f_n\} \subset L^p(X,\mu)$ be a sequence of integrable functions. The sequence $\{f_n\}$ is said to be uniformly integrable if for every $\varepsilon>0$, there exists $\delta>0$ such that for any measurable set $E \subset X$ with $\mu(E)<\delta$, one has
\[
\sup_{n} \int_E |f_n|^p\, d\mu < \varepsilon.
\]
\end{definition}

\begin{Theorem}[ Vitali Convergence Theorem, \texorpdfstring{\cite[Theorem 16.6]{Schilling2005}}{}]
Let $\{f_n\} \subset L^p(\mu)$ be a uniformly integrable sequence of integrable functions that converges almost everywhere to a function $f$. Then $f \in L^p(\mu)$ and
\[
\lim_{n \to \infty} \int_X |f_n - f|^p\, d\mu = 0.
\]
\end{Theorem}

\section{Proof of Theorem \ref{th:Gro1} and Theorem \ref{th:compact}}

\begin{proof}[Proof of Theorem \ref{th:compact}]
Let $\{f_k\} \subset A^q(\Omega)$ be a bounded sequence and $M > 0$ a constant such that
\[
\sup_k \|f_k\|_q \leq M.
\]
By H\"older's inequality, it follows that
\begin{align*}
\int_{\Omega} |f_k|^p &\leq \left( \int_{\Omega} 1 \right)^{1 - p/q} \left( \int_{\Omega} |f_k|^q \right)^{p/q}\\
&\leq \left( \int_{\Omega} 1 \right)^{1 - p/q} \cdot\|f_k\|^p_q\\
&\leq |\Omega|^{1-p/q}\|f_k\|_q^p,
\end{align*}		
i.e.,
\[
\sup_k \|f_k\|_p \leq |\Omega|^{1/p-1/q}\|f_k\|_q\leq CM.
\]
Here, $C$ is a constant depending on $p,q$ and the volume of $\Omega$. On the other hand, for any fixed compact subset $K \subset\subset \Omega$, the Bergman inequality yields
\[
\sup_{z \in K} |f_k(z)| \leq C_{K,\Omega}^{1/q} \|f_k\|_q \leq C_{K,\Omega}^{1/q} M, \quad \forall\ k,
\]
where $C_{K,\Omega}>0$ depends only on $K$ and $\Omega$. Thus $\{f_k\}$ is locally uniformly bounded on $\Omega$, and Montel's theorem implies that it forms a normal family. Consequently, there exists a subsequence $\{f_{k_j}\} = \{g_j\}$ that converges uniformly on compact subsets of $\Omega$ to some $g \in \mathcal{O}(\Omega)$.

In order to show convergence in the sense of $L^p$-norm, we invoke Vitali's convergence theorem. Since $\sup_j \|g_j\|_q \leq M$, for any measurable set $E \subset \Omega$, H\"older's inequality gives
\[
\int_E |g_j|^p \leq |E|^{1-p/q}\left(\int_E |g_j|^q\right)^{p/q}\leq V(E)^{1-p/q}M^p.
\]
Thus, for any $\varepsilon > 0$, if $V(E)$ is sufficiently small, then $\int_E |g_j|^p < \varepsilon$, uniformly in $j$. By Vitali's convergence theorem, we conclude that
\[
\|g_j - g\|_p \to 0 \quad \text{as } j \to \infty,
\]
showing that $g_j \to g$ in $A^p(\Omega)$, which proves that $\tau$ is compact.		
\end{proof}

Note that the natural inclusion operator $j$ defined in \eqref{eq:natural_inclusion} can be decomposed as follows:
\begin{equation}\label{eq:inclusion_decomposition}
(E,\|\cdot\|_p)\xrightarrow{i} A^q(\Omega)\xrightarrow{\tau} A^p(\Omega).
\end{equation}
Here, $i$ denotes the inclusion operator from $E$ into $A^q(\Omega)$, where $A^q(\Omega)$ is endowed with the $L^q$-norm. Next, we shall prove

\begin{lemma}\label{lm:i-continuity}
The inclusion operator
\[
i : (E,\|\cdot\|_p) \longrightarrow A^q(\Omega), \qquad i(f) = f
\]
is a bounded linear operator.
\end{lemma}	
\begin{proof}
Note that both $(E, \|\cdot\|_p)$ and $A^q(\Omega)$ are Banach spaces. Consider the linear operator
\[
i : E \longrightarrow A^q(\Omega), \quad i(f) = f,
\]
whose graph is given by
\[
\{(f, f)\ \big|\  f \in E\}.
\]

We now verify that the graph of $i$ is closed, that is, if $f_j\in E$ with $f_j\to f$ in the sense of $L^p$-norm and $i(f_j)=f_j\to g$ in the sense of $L^q$-norm, then $f=g$. More precisely, $\{f_j\} \subset E$ satisfies
\[
\|f_j - f\|_p \to 0 \quad \text{and} \quad \|f_j - g\|_q \to 0 \quad \text{as } j \to \infty,
\]
for some $f \in E$ and $g \in A^q(\Omega)$. Then, for any compact subset $K \subset \Omega$, the Bergman inequality yields
\begin{align*}
\sup_{z \in K} |f_j(z) - f(z)| &\le C_{K,\Omega}^{1/p} \|f_j - f\|_p,\\
\sup_{z \in K} |f_j(z) - g(z)| &\le C_{K,\Omega}^{1/q} \|f_j - g\|_q,
\end{align*}
where $C_K > 0$ depends only on $K$. Letting $j \to \infty$, we conclude that $\{f_j\}$ converges locally uniformly both to $f$ and to $g$, which implies $f = g$ on $K$. Since $K$ is arbitrary, it follows that $f \equiv g$ on $\Omega$.

Since the graph of $i$ is closed, the closed graph theorem implies that $i$ is bounded.
\end{proof}
	
\begin{proof}[Proof of Theorem \ref{th:Gro}]
Since the composition of compact operator with a bounded operator is compact, we infer from \eqref{eq:inclusion_decomposition}, Theorem \ref{th:compact} and Lemma \ref{lm:i-continuity} that $j$ is compact. The conclusion follows directly from the Theorem \ref{functional analysis}.
\end{proof}


\end{document}